\ifcsname ifsmall\endcsname\else
  \expandafter\let\csname ifsmall\expandafter\endcsname
                  \csname iffalse\endcsname
\fi

\ifsmall
\documentclass[oneside]{amsart}
\usepackage[paperwidth=169mm,paperheight=239mm]{geometry}
\else
\documentclass{amsart}
\fi

\usepackage{graphicx}

\renewcommand{\epsilon}{\varepsilon}

\newcommand{\R}{\mathbb{R}}
\newcommand{\T}{\mathbb{T}}

\newcommand{\dd}{\,d}
\newcommand{\nl}{\left\lVert}
\newcommand{\nr}{\right\rVert}

\renewcommand{\limsup}{\varlimsup}
\renewcommand{\liminf}{\varliminf}

\newcommand{\seq}[1]{(#1_k)_{k\geq 0}}
\newcommand{\M}{\mathcal{M}}
\newcommand{\product}[1]{\mu_{\M}}
\newcommand{\norm}[1]{\nl #1\nr_{\M}}
\renewcommand{\S}{\mathcal{S}}
\newcommand{\Scomp}{\widehat{\S}}
\newcommand{\C}{\mathcal{C}}
\newcommand{\Ccomp}{\widehat{\C}}

\newcommand{\ST}{\S_\T}

\newcommand{\normT}[1]{\norm{#1}}
\newcommand{\supp}{\operatorname{supp}}

\renewcommand{\and}{\hspace{8pt} \text{and} \hspace{8pt}}
\newcommand{\q}{q}

\newcommand{\tM}{\widetilde{M}}

\newtheorem{thm1}{Theorem}[section]
\newtheorem{theorem}[thm1]{Theorem}
\newtheorem{lemma}[thm1]{Lemma}
\newtheorem{corollary}[thm1]{Corollary}
\newtheorem{proposition}[thm1]{Proposition}

\theoremstyle{definition}
\newtheorem{definition}[thm1]{Definition}
\newtheorem{conjecture}{Conjecture}

\theoremstyle{remark}
\newtheorem*{remark}{Remark}


\title{Carleman-Sobolev classes for small exponents}
\author{Gustav Behm}\thanks{The research of the first named author was supported by H{\aa}kan Hedenmalm's grant from the Swedish Research Council (VR: 2012-3122).}
\author{Aron Wennman}

\email[Gustav Behm]{gbehm@math.kth.se}
\email[Aron Wennman]{aronw@math.kth.se \textrm{(Corresponding author)}}
\address{Department of Mathematics, KTH Royal Institute of Technology,\newline \indent SE-100 44 Stockholm, Sweden}

\date{\today}

\begin{document}

\begin{abstract}
This paper is devoted to the study of a generalization of Sobolev spaces for small $L^{p}$ exponents, i.e.~$0<p<1$. We consider spaces defined as abstract completions of certain classes of smooth functions with respect to weighted quasi-norms involving derivatives of all orders, simultaneously inspired by Carleman classes and classical Sobolev spaces. If the class is restricted with a growth condition on the supremum norms of the derivatives, we prove that there exists a condition on the weight sequence which guarantees that the resulting space can be embedded into $C^{\infty}(\R)$. This condition is necessary for such an embedding to be possible, up to some regularity conditions on the weight sequence. We also show that the growth condition is necessary, in the sense that if we drop it entirely we can naturally embed $L^p$ into the resulting completion.
\end{abstract}

\maketitle

\setcounter{section}{-1}
\section{Introduction}
The purpose of this paper is to introduce a class of functions, simultaneously inspired by Sobolev spaces on the real line and so-called Carleman classes from the study of quasi-analytic functions. This shall be done for small $L^p$-exponents, that is, $p$ with $0<p<1$. In the more usual case, when $p\geq1$, one can define Sobolev spaces in several equivalent ways. In the case $0<p<1$, however, it turns out that these definitions are not equivalent. In this paper we will work with the definition of Sobolev spaces as abstract completions of a set of smooth functions with respect to some norm.

In \cite{Peetre} Peetre chooses to define Sobolev spaces in this way, using the Sobolev norm
$$
\nl f\nr_{k,p}=\left(\lVert f\rVert_{p}^{p}+\lVert f'\rVert_{p}^{p}+\ldots+\lVert f^{(k)}\rVert_{p}^{p}\right)^{1/p}.
$$
Although this seems quite natural, it soon turns out that this leads to pathology. Already Douady showed that if the space $W^{1,p}$ were to be possible to view as a function space, one would have {\em functions}, for a lack of a better word, which are identically zero, but with a derivative that is equal to one almost everywhere!
There is no original reference for this, but the Douady example is described in detail in \cite{Peetre} 
Peetre \cite{Peetre} went on to show that the situation is even worse. Completely contradictory to the case when $p\geq1$ one can embed $L^p$ naturally into $W^{k,p}$. Moreover there is actually an isomorphism
$$
W^{k,p}\cong L^{p}\oplus L^{p}\oplus \ldots \oplus L^{p}\cong L^{p},
$$
meaning that in the completion there is a complete uncoupling between a function and its different derivatives. It is worth noting that by a classical theorem of \mbox{Day \cite{Day}} which characterizes the dual of $L^p$, the dual of $W^{k,p}$ is actually trivial, so these spaces cannot even be regarded as spaces of distributions. Peetre was led to consider Sobolev spaces for $0<p<1$ by considering a problem in non-linear approximation theory. More precisely, he asked when best approximation by spline functions of a given order is possible. For a further discussion of this topic, see \cite[Chapter~11]{PeetreBook}.

The other kind of function classes that we have mentioned are the Carleman classes. For a weight sequence $\M=\seq{M}$, a function $f\in C^{\infty}$ is said to belong to the Carleman class $C\{M_k\}$ if there exists some $b>0$ such that
$$
\frac{\nl f^{(k)}\nr_{\infty}}{M_k}\leq b^k, \quad k\geq0.
$$
The constant $b$ contributes in some sense to a softening of the topology of $C\{ M_k\}$. While being important in some respects, for many purposes one can instead study $g(x)=f(b^{-1}x)$, for which we get $\nl g^{(k)}\nr_{\infty}=b^{-k}\nl f^{(k)}\nr_{\infty}$. We could thus study $f$ for which
$$
\sup_{k\geq0}\frac{\nl f^{(k)}\nr_{\infty}}{M_k}\leq 1.
$$
When $p\geq1$ there is a connection between the supremum of a function $f$ and the $L^p$-norms of $f$ and $f'$, at least it is clear for compactly supported $f$. Thus one can sometimes substitute the supremum norms for $L^p$-norms and end up with very similar classes. For further discussion on this, see \cite[Chapter~V]{Katznelson}.

One of our starting points is the question regarding what happens to $W^{k,p}$ if we let $k$ tend to $\infty$ in the presence of a weight sequence $\M$. We shall fix an exponent $p$ with $0<p<1$ and study the completion of different classes of $C^{\infty}$-smooth test functions with respect to the quasi-norm
$$
\nl f\nr_{\M}=\sup_{k\geq0}\frac{\nl f^{(k)}\nr_{p}}{M_k}.
$$
This is actually more inspired by expressions encountered in works related to Carleman classes, but for some purposes this will be interchangeable with the more Sobolev-flavored weighted expression
$$
\left(\frac{\lVert f\rVert_{p}^{p}}{M_0^{p}}+\ldots+\frac{\lVert f^{(k)}\rVert_{p}^{p}}{M_k^{p}}+\ldots \right)^{1/p}.
$$
This Carleman-type norm is chosen in the hope of making computations tractable, while retaining the essence of a Sobolev space of infinite order. It is because of these considerations we choose the nomenclature {\em Carleman-Sobolev classes}.

In this paper we present three results regarding such completions. We show that the pathology encountered for $W^{k,p}$ largely remains if we consider the completion $\Ccomp$ of the subset $\C$ of $C^{\infty}$-smooth functions with $\norm{f}<\infty$. Indeed, we find a continuous embedding $L^{p}(\R)\hookrightarrow \Ccomp$. This holds no matter what restrictions we put on the sequence $\M$.

The second and third results concerns the completion of the smaller class $\S$, consisting of smooth functions with finite norm $\norm{f}<\infty$ whose derivatives satisfy the the growth condition
$$
\limsup_{k\to\infty} \nl f^{(k)}\nr_{\infty}^{\q^k}\leq 1,
$$
where $q=1-p$ to simplify notation. Provided that $\M$ does not grow too quickly, expressed as
$$
\product{M}:=\prod_{k=0}^{\infty}M_k^{\q^k}<\infty
$$
the situation turns out to be completely different. In this case we can find a continuous embedding $\Scomp \hookrightarrow C(\R)$, and assert that $\Scomp$ can be regarded as a subset of $C^{\infty}$. These conditions were observed by Hedenmalm and communicated to the authors. They turn up when employing an iterative scheme in the attempt to bound the supremum norm of a function by the $p$-norms of the function and its derivatives in the proof of Theorem~\ref{completion-is-smooth} below. Hedenmalm invented the scheme after using similar techniques in work with Borichev \cite{HedenmalmBorichev}. It is inspired by the proof of an inequality due to Hardy-Littlewood, see Garnett's book \cite[Lemma~3.7]{Garnett}.

This condition is somewhat akin to the celebrated condition
$$\sum_{k=0}^\infty \frac{M_k}{M_{k+1}} = \infty$$
which by the Denjoy-Carleman Theorem \cite{Cohen} is the precise condition when $C\{M_k\}$ is a quasi-analytic class.

There seems to be a similar sharpness regarding the condition $\product{M}<\infty$. Under certain regularity assumptions on $\M$ we show that if $\product{M}=\infty$ such a continuous embedding is impossible. That is, there can be no constant $C$ for which
$$
\nl f\nr_{\infty}\leq C\nl f\nr_{\M},\quad f\in\S.
$$

This paper is organized as follows. In Section~$1$ we collect relevant definitions and preliminary results. In Section~$2$ we study the space $\Scomp$ when $\product{M}<\infty$. Section~$3$ is devoted to investigating the conditions defining $\S$: $\product{M}<\infty$ and \mbox{$\limsup \nl f^{(k)}\nr_{\infty}^{\q^k}\leq 1$}. Under certain additional assumptions we manage to show that $\product{M}<\infty$ is necessary for an embedding into $C(\R)$ to exist. We then move on to study the more degenerate case which one gets from the completion of the bigger class $\C$. In Section~$4$ we extend some of the results to the setting of the unit circle $\T$, and we end the paper with a discussion and two conjectures in Section~$5$.

We would like to thank our advisor H{\aa}kan Hedenmalm for suggesting this research topic.


\section{Preliminaries and definitions}
Fix a number $p$ with $0<p<1$ and its corresponding number $q=1-p$. This number will serve as exponent for the $L^{p}$-quasi-norm which as usual is defined as
$$
\nl f\nr_{p}=\left(\int_{\R}\lvert f(x)\rvert^{p}\dd x\right)^{1/p}
$$
for measurable functions $f$ on the real line $\R$. The topological vector space consisting of all measurable functions $f$ such that $\nl f\nr_{p}<\infty$ is denoted by $L^{p}(\R)$. When $p\geq$ this expression is a norm and $L^{p}(\R)$ is a Banach space. In our case, however, it is merely a quasi-norm and the space is a quasi-Banach space. By a quasi-norm we mean that $\nl \cdot \nr_{p}$ is a homogeneous, positive definite real-valued function such that a quasi-triangle inequality, i.e. the inequality
$$
\nl f+g\nr_{p}\leq K\left(\nl f\nr_{p}+\nl g\nr_{p}\right), \quad f,g\in L^{p}(\R)
$$
holds only for some constant $K>1$. The best such constant is quite readily seen to be $2^{\frac{1-p}{p}}$.
However the ordinary triangle inequality still remains valid in the following sense:
\begin{equation}\label{triangle}
\nl f+g\nr_p^p\leq \nl f\nr_{p}^p+\nl g\nr_{p}^p, \quad f,g\in L^{p}(\R).
\end{equation}

As usual let $C^{k}(\R)$ denote the spaces of $k$ times continuously differentiable functions on $\R$, and $C^{\infty}(\R)$ is the intersection of all these. When necessary we will consider derivatives to be taken in the sense of distributions.

We will mostly be working with spaces whose topology is induced by the quasi-norm mentioned in the introduction:
\begin{definition}
For a fixed sequence $\M=\seq{M}$ of numbers $M_k\geq1$ we define the quasi-norm
\begin{equation}\label{Mknorm}
\norm{f}=\sup_{k\geq0} \frac{\nl f^{(k)}\nr_p}{M_k}, \quad f\in C^{\infty}(\R).
\end{equation}
\end{definition}

Observe that if for two sequences $\M$ and $\mathcal{N}$ are such that $\widetilde{M}\leq\mathcal{N}$ entrywise, we will have the norm inequality
\begin{equation}\label{norm-ineq}
\nl u\nr_{\M}\leq \nl u\nr_{\mathcal{N}}.
\end{equation}

As mentioned, this quasi-norm will be used to define spaces of functions which when completed become quasi-Banach spaces. We will be working with test classes where the derivatives can be taken in the ordinary sense so that the quasi-norm \eqref{Mknorm} has a clear meaning, and their abstract completion with respect to the induced topology. The two classes that we will be concerned with are the following.

\begin{definition}
Let $\C$ and $\S$ be the classes of functions defined by
$$\C = \left\{f\in C^{\infty}(\R) : \norm{f}<\infty\right\} $$
and
$$\S = \left\{f\in C^\infty(\R) : \norm{f}<\infty \and \limsup_{k\to\infty} \nl f^{(k)}\nr_\infty^{\q^k} \leq 1 \right\},$$
respectively. Note that both classes are vector spaces and that they depend on the number $0<p<1$ and the choice of the sequence $\M$. We will denote the abstract completions of $\C$ and $\S$ by $\Ccomp$ and $\Scomp$, respectively.
\end{definition}

Again, if we have $\mathcal{N}\leq\M$ then the the class defined by the smaller sequence $\mathcal{N}$ is a subset of the corresponding class for the larger sequence, by them norm inequality \eqref{norm-ineq}. We observe also that $\S\subset \C$.

The condition $\limsup_{k\to\infty}\nl f^{(k)}\nr_{\infty}^{\q^{k}}\leq 1$ may seem strange, to understand where it comes from we refer to the proof of Proposition~\ref{sup-est} below. Note also that it is not very restrictive; we have $0<p<1$ so also $0<\q<1$, which means that $\q^k$ tends to zero exponentially quick.

The following is an easy consequence of the definitions, and provides the reason for thinking of $\Scomp$ and $\Ccomp$ as kinds of Sobolev spaces, explaining their relation to the $L^p$-spaces.

\begin{lemma}
\label{l:ssublp}
If $f\in\S$ or $f\in\C$ then $f^{(k)}\in L^p(\R)$ for any $k=0,1,\ldots$.
\end{lemma}

\begin{proof}
If $f$ belongs to any one of these classes then $\norm{f}<\infty$ and therefore
\begin{equation*}
\nl f^{(k)}\nr_p =M_k \frac{\nl f^{(k)}\nr_{p}}{M_k} \leq M_k \sup_{l\geq0} \frac{\nl f^{(l)}\nr_p}{M_l} = M_k \norm{f} < \infty.\qedhere
\end{equation*}
\end{proof}

\section{Smooth embedding of $\Scomp$}
This section is devoted to the study of the completion $\Scomp$ of the class $\S$ with respect to the norm $\norm{\cdot}$, in the case when the sequence $\M$ does not grow too quickly. This is expressed as

\begin{equation}\label{mkbound}
\product{M}:=\prod_{k\geq 0}M_{k}^{p\q^{k-1}}<\infty.
\end{equation}
and will be assumed throughout this section. Our first main result is the following.

\begin{theorem}\label{completion-is-smooth}
Assume that the sequence $\M$ satisfies $\product{M}<\infty$. Then $\Scomp$ can be canonically and continuously embedded in $C^{\infty}(\R)$.
\end{theorem}

We will present the proof of this claim towards the end of this section. First we proceed with a preliminary result and a simple corollary thereof. The proof of the below proposition is interesting in its own right since it explains where the conditions \eqref{mkbound} and $\limsup \nl f^{(k)}\nr_{\infty}^{\q^k}$ comes from. We remark that this result is due to H{\aa}kan Hedenmalm. Since it has not been published, we present his result here with a proof.

\begin{proposition}[Hedenmalm, \cite{private}]
\label{sup-est}
If $f\in \S$ then
$$\nl f\nr_\infty \leq \product{M} \norm{f}.$$
\end{proposition}

This was inspired by work done in \cite{HedenmalmBorichev} where it was used to establish what is termed {\em Hardy-Littlewood ellipticity} of the $N$-Laplacian, originally studied by Hardy and Littlewood in \cite{Hardy1932}. This is a rather remarkable feature --- for a harmonic function $u$ and $p\geq1$ it is natural to expect that $\lvert u(z_0)\rvert$ can be controlled by the $p$-norm of $u$, i.e. that
$$
\lvert u(z_0)\rvert \leq \frac{1}{\pi}\int_{\mathbb{D}(z_0,1)}\lvert u\rvert^{p}\dd A(z).
$$
This is due to subharmonicity of $\lvert u\rvert^{p}$. When $0<p<1$ the function $z\mapsto \lvert u(z)\rvert^{p}$ is no longer subharmonic, but the bound of $\lvert u(z_0)\rvert$ in terms of an area integral survives
nevertheless, although with a different constant. The proof is loosely based on similar ideas.

\ifsmall\else
\newpage 
\fi

\begin{proof}
We assume that $\norm{f}=1$.
Since $f\in C^\infty(\R)$ we can write
$$f(x)-f(y) = \int_y^x f' \dd t$$
and use this to estimate
$$|f(x)-f(y)|\leq \int_y^x |f'| \dd t = \int_y^x |f'|^p |f'|^{1-p}\dd t \leq \nl f'\nr_\infty^{1-p} \nl f'\nr_p^p.$$
By Lemma~\ref{l:ssublp} we have $f\in L^p(\R)$. Thus for some value of $y$, $f(y)$ will be arbitrarily small.
Therefore if we apply the reverse triangle inequality and take the supremum over $x$ we get
$$\nl f\nr_\infty \leq \nl f'\nr_\infty^{\q} \nl f'\nr_p^p.$$
Repeating this estimate for $f'$ instead of $f$ and using it in the previous estimate we find that
$$\nl f\nr_\infty \leq \nl f''\nr_\infty^{\q^2} \nl f''\nr_p^{p\q}  \nl f'\nr_p^p.$$
Iterating this we end up with
$$\nl f\nr_\infty \leq \nl f^{(n)}\nr_\infty^{\q^n} \times\prod_{k=1}^n \nl f^{(k)}\nr_p^{p\q^{k-1}} .$$
Now by Lemma \ref{l:ssublp} and our assumption $\norm{f}=1$
$$\nl f^{(k)}\nr_p^{p\q^{k-1}} \leq M_k^{p\q^{k-1}},$$
and we arrive at
$$
\nl f\nr_\infty \leq \nl f^{(n)}\nr_\infty^{\q^n}\times\prod_{k=1}^n M_k^{p\q^{k-1}}.
$$
Since $\limsup \nl f^{(n)}\nr_{\infty}^{\q^{n}}\leq 1$ we can let $n$ tend to $\infty$ to obtain
\begin{equation*}
\nl f\nr_\infty \leq \prod_{k=1}^\infty M_k^{p\q^{k-1}}=\product{M}.\qedhere
\end{equation*}
\end{proof}

The following corollary is a simple extension of this technique.

\begin{corollary}\label{derivative-sup-est}
If $f\in\S$ then
$$
\nl f^{(i)}\nr_{\infty}\leq \product{M}^{1/\q^{i}}\norm{f}, \quad i=0,1,\ldots.
$$
\end{corollary}

\begin{proof}
Again, assume that $\norm{f}=1$ and use the same argument as in Proposition \ref{sup-est} but starting with $f^{(i)}$ instead of $f$.
Hence we can get the estimate
$$\nl f^{(i)}\nr_\infty \leq \nl f^{(i+1)}\nr_\infty^{\q} \nl f^{(i+1)}\nr_p^p.$$
As before we iterate this estimate and using that $\norm{f}=1$ we have
$$\nl f^{(k+i)}\nr_p^{p\q^{k-1}} \leq M_{i+k}^{p\q^{k-1}}.$$
Hence
$$\nl f^{(i)}\nr_\infty \leq \nl f^{(n+i)}\nr_\infty^{\q^n} \times \prod_{k=1}^n M_{i+k}^{p\q^{k-1}}\leq \nl f^{(n+i)}\nr_\infty^{\q^n}\product{M}^{1/\q^i}$$
where in the last inequality we use
$$
\prod_{k=1}^n M_{i+k}^{p\q^{k-1}} = \left(\prod_{k=1}^n M_{i+k}^{p\q^{i+k-1}}\right)^{1/\q^i}\!\!\!\!\leq\left(\prod_{l=1}^\infty M_l^{p\q^{l-1}}\right)^{1/\q^i} \!\!\!\! =\product{M}^{1/\q^i}.
$$
A similar trick gives
$$\limsup_{n\to\infty} \nl f^{(n+i)}\nr_\infty^{\q^n} \leq 1,$$
and as in the previous proposition, we let $n\to\infty$ to obtain the desired estimate.
\end{proof}

We proceed with the proof of the main theorem of this section using this corollary.

\begin{proof}[Proof of Theorem \ref{completion-is-smooth}]
Assume that $f\in \Scomp$. Then it can be represented by a Cauchy sequence $\{f_{n}\}$ in $\S$ and by Lemma \ref{sup-est} we have
$$
\nl f_{n}-f_m\nr_{\infty}\leq \product{M} \norm{f_n-f_m}\to 0,
$$
so $\{f_n\}$ is Cauchy in supremum norm as well, which implies that $f_n\to g$ for a (unique) continuous $g$.

Due to Corollary \ref{derivative-sup-est} the similar estimate
\begin{equation}
\nl f_{n}^{(i)}-f_m^{(i)}\nr_{\infty}\leq \product{M}^{1/\q^{i}} \norm{f_n-f_m}
\end{equation}
will hold for any derivative $f^{(i)}, i=0,1,\ldots$ of $f$.
Therefore there will exist (unique) functions $g_i\in C(\R)$ such that $f_n^{(i)}\to g_i$ in supremum norm for each $i\geq1$. It is clear that $g_i=g^{(i)}$.
Thus $g\in C^{\infty}(\R)$ is the limit of the sequence $\{f_n\}$; a representative of $f\in\Scomp$.
We thus define a mapping by $f\mapsto g$ under these circumstances, and injectivity, linearity and continuity for this embedding are all readily verified.
\end{proof}

\section{Converse results for $\Scomp$ and $\Ccomp$}
In this section we shall next try to understand to what extent Theorem~\ref{completion-is-smooth} is sharp. We first ask what happens if we drop the condition
\begin{equation*}\label{mu}
\product{M}=\prod_{k=0}^{\infty}M_{k}^{p\q^{k-1}}<\infty
\end{equation*}
which together with the supremum-norm growth condition
\begin{equation}\label{liminf}
\limsup_{n\to\infty}\nl f^{(n)}\nr_{\infty}^{\q^{n}}\leq 1
\end{equation}
gave such nice results for $\Scomp$. Will we still end up in $C^{\infty}(\R)$, and if not --- will there be some kind of phase transition as the product $\product{M}$ becomes infinite? We make contrast between the situation in Theorem \ref{completion-is-smooth}, where we have an inequality
$$
\nl f\nr_{\infty}\leq C\norm{f}, \quad f\in \S
$$
(and a similar one for the derivatives of $f$) and thus are able to define a continuous natural embedding $\Scomp\hookrightarrow C^{\infty}(\R)$, and the situation where such an inequality is impossible. Note that if this inequality fails, it impossible to embed $\Scomp$ even into $C(\R)$ endowed with supremum norm.

Towards the end of the section, we investigate what happens if one instead drops condition \eqref{liminf} and thus considers the completion of the class $\C$.
\subsection{Construction of smooth functions by infinite convolutions}

Building on a discussion by H\"ormander in \cite{HORM1}, where he constructs so-called mollifiers using decreasing sequences $\seq{a}\in \ell^{1}(\mathbb{N})$ of positive numbers, we set for each $n\geq0$
$$
H_{a_n}(x)=\frac{1}{a_n}\chi_{[0,a_n]}(x),\quad x\in\R.
$$
We then define a sequence of functions
\begin{equation}\label{unconstruct}
u_n=H_{a_0}*H_{a_1}*\ldots *H_{a_n},\quad n\geq0,
\end{equation}
and note that $u_n$ has support on $[0, a_0+\ldots a_n]$. Our intention is to let $n$ tend to infinity to obtain a smooth function which hopefully will be a member of one of our classes $\S$ or $\C$. These limit functions will not, however, automatically lend themselves to any nice computations, as is the case with finite $n$. We recall that $u_n\in C^{n}(\R)$ and that $u_n\to u$ uniformly on $\R$ for some $u\in C^{\infty}(\R)$, supported on $[0,\sum_j a_j]$ which is a compact set since $(a_k)_{k=0}^{\infty}\in\ell^{1}$. Our goal is to provide conditions so that expressions involving $u$ can be compared nicely to corresponding ones for $u_{n}$ instead.

We remark that these functions are usually referred to as {\em spline functions}, which are piecewise polynomials with finitely many break points. This topic is, however, outside the field of expertise of the authors, and for a treatment of such functions and their role in constructive approximation theory see for example the text books \cite{DeVore} and \cite{Ahlberg}. 

To get some feel for these functions $u_n$, consider $u_1(x)=H_{a_0}*H_{a_1}(x)$, where $a_1<a_0$. By direct computation we find that it is a tent-like function with a plateau:
$$
u_1(x)=\begin{cases}  \frac{x}{a_0a_1} & 0\leq x \leq a_1\\ \frac{1}{a_0} & a_1 < x \leq a_0 \\ \frac{a_0+a_1-x}{a_0a_1}& a_0 < x \leq a_0+a_1\end{cases}
$$
The function $u_1$ is illustrated in the middle column of Figure~\ref{unk}. A differentiation of the above yields
$$
u_1'(x)=\begin{cases}\frac{1}{a_0a_1} & 0\leq x\leq a_1\\ 0 & a_1< x\leq a_0\\ -\frac{1}{a_0a_1}& a_0< x\leq a_0+a_1\end{cases}
$$
almost everywhere. Thus $u'_{1}$ has a very similar shape as $u_0:=H_{a_0}$, see the diagonal plots in Figure~\ref{unk}. To construct $u_2$ we convolve again by $H_{a_2}$. If we want to keep the symmetry, i.e. if we want $u'_{2}$ to have a similar as $u_1$ and  $u''_2$ to resemble $u'_1$ (i.e., no overlap in the top derivatives $u_n^{(n)}$), we will need to require $a_2+a_1<a_0$ and $a_2<a_1$.

\begin{figure}[t]
    \centering
    \includegraphics[width=0.8\textwidth]{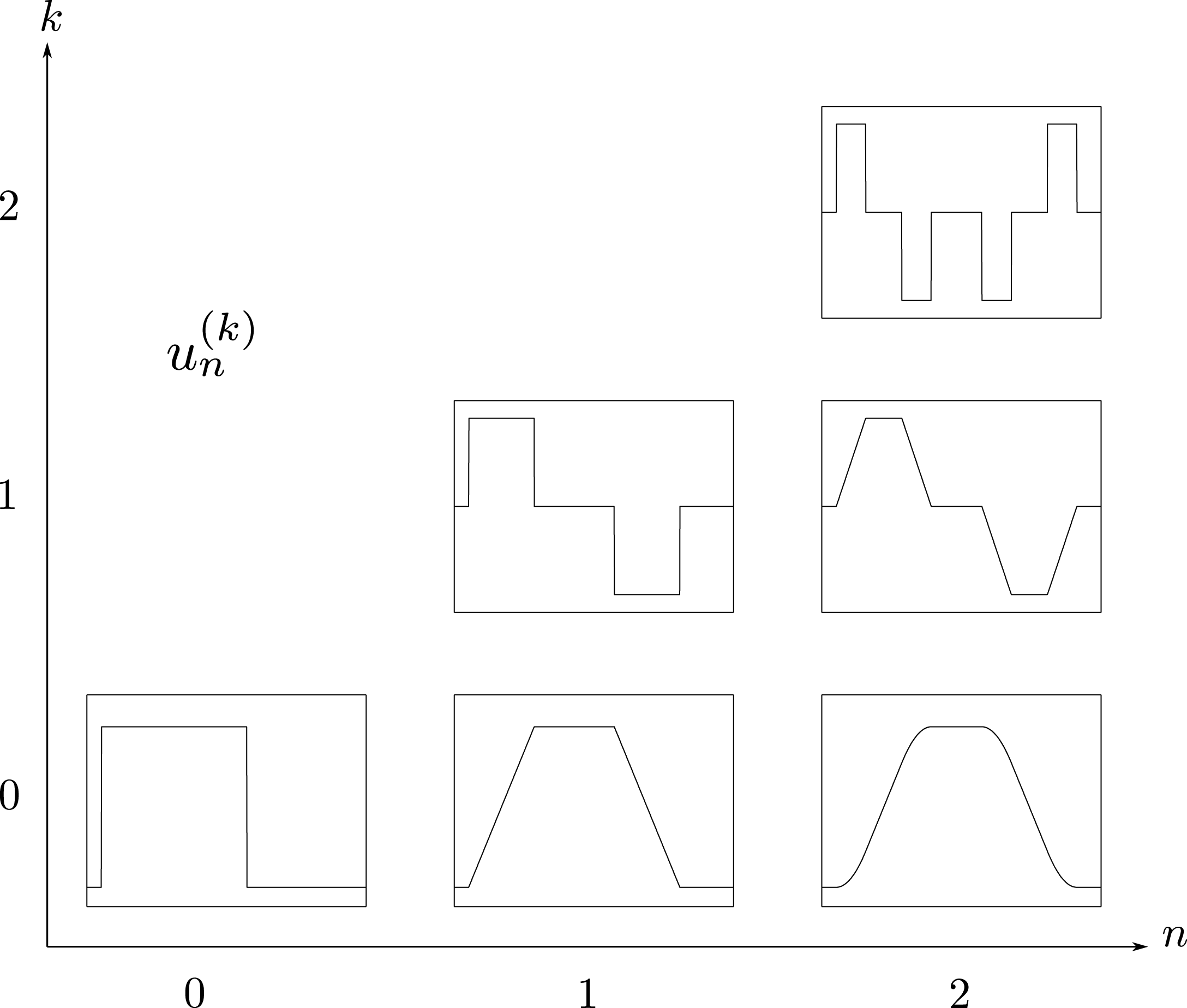}
    \caption{The appearance of $u_{n}^{(k)}$}
    \label{unk}
\end{figure}

It turns out that the condition
\begin{equation}\label{separation}
\sum_{j>k}a_j\leq(1-c)a_k,\quad k\geq0
\end{equation}
for some $c\in(0,1)$ is the appropriate generalization of the requirement $a_1<a_0$, when $k>1$. If we want the different translates of the characteristic functions that makes up the top derivatives $u_{k}^{(k)}$ to be disjoint, this is what does the trick. The results alluded to above are captured by the following lemma.

\begin{lemma}\label{convol-basic}
Let $a=(a_n)_{n\geq0}\in\ell^{1}$ denote a decreasing sequence of positive numbers that satisfies \eqref{separation}. Then $u_{\:n}^{(n)}$ is supported on $2^n$ disjoint intervals $I_j$, \mbox{$j=1,2,\ldots, 2^n$} of length $a_n$ and on each such interval we have
$$
\lvert u_{\:n}^{(n)}(x)\rvert =\frac{1}{a_0a_1\cdot \ldots \cdot a_n}, \quad x\in I_j.
$$
Moreover, the leftmost interval $I_1$ is $[0,a_n]$, and the others are obtained by different translations.
\end{lemma}

\begin{proof}
We proceed by induction. The lemma is trivially true for $n=0$. Suppose the lemma is true for some $n\geq1$. That is, whenever a sequence $a$ satisfies the hypotheses, the resulting function $u_{\:n}^{(n)}$ will have the desired properties. Consider such a sequence $a=(a_j)_{j=0}^{\infty}$ and the corresponding $u_{\:n+1}^{(n+1)}$. It holds that
$$
u_{\:n+1}^{(n+1)}(x)=\frac{d}{dx}H_{a_0}*\frac{d^n}{dx^n}H_{a_1}*\ldots *H_{a_{n+1}}(x)=\frac{1}{a_0}(1-\tau_{a_0})\tilde{u}_{\:n}^{(n)}(x),
$$
where $\tilde{u}_n=H_{a_1}*\ldots* H_{a_n}$ and where $\tau_{a}$ denotes translation by $a$. Now we want to apply the induction hypothesis to $\tilde{u}_n$. To be able to do this we need to verify that the properties of $\seq{a}$ are inherited by the shifted sequence $\seq{b}$, where $b_k=a_{k+1}$. This is no problem however, since $\seq{b}$ is clearly a decreasing sequence of positive numbers, and
$$
\sum_{j>k}b_j=\sum_{j>k}a_j+1 \leq (1-c)a_{k+1}=(1-c)b_k, \quad k\geq0.
$$
Thus we can infer that $\tilde{u}_n^{(n)}$ has the desired properties, so we can write
$$
\tilde{u}_{\:n}^{(n)}(x)=\frac{1}{a_1a_2\cdots a_{n+1}}\sum_{j=1}^{2^n}\pm \chi_{I_j'}(x),\quad x\in\R.
$$
To calculate $u^{(n+1)}_{n+1}$ we want to subtract a translated copy of $\tilde{u}^{(n)}_{n}$ to itself. The supports of these copies do not overlap since $\tilde{u}^{(n)}_{n}$ is supported on 
$$[0,a_1+a_2+\ldots+a_{n+1}]\subset [0,a_0].$$ 
From this it follows that the lemma holds for $n+1$ and the proof is complete.
\end{proof}

We can do even better. Using the previous ideas we can prove the following proposition, which permits us to move calculations from $u^{(n)}$ to $u^{(n)}_{n}$ when calculating $L^{p}$-norms.
\begin{proposition}\label{convol-est} Let $\seq{a}$ denote a sequence satisfying \eqref{separation}. Then
$$
\nl u^{(n)}\nr_{\infty} = \nl u_{n}^{(n)}\nr_{\infty}=\frac{1}{a_0a_1\cdots a_n}, \quad n\geq0,
$$
and we have the estimate
$$
\nl u^{(n)}(x)\nr_{p}\leq (2-c)^{1/p}\nl u_{\:n}^{(n)}\nr_{p}=(2-c)^{1/p}\frac{2^{n/p}a_n^{1/p}}{a_0\cdot \ldots \cdot a_n},\quad n\geq0.
$$
\end{proposition}

\begin{proof}
For the first assertion, just observe that for any $k\geq1$
$$
\nl u_{\:n+k}^{(n)}\nr_{\infty}=\nl u_{n}^{(n)}*H_{n+1}*\ldots*H_{n+k}\nr_{\infty}\leq \nl u_{\:n}^{(n)}\nr_{\infty}
$$
since $\int H_{n+1}*\ldots*H_{n+k}=1$. Now since $u^{(n)}=\lim_k u_{\:n+k}^{(n)}$ it follows that we have the inequality \mbox{$\nl u^{(n)}\nr_{\infty}\leq \nl u_{\:n}^{(n)}\nr_{\infty}$}. To prove the reverse inequality, we shall use induction on $k$ to find an $x$ such that $u_n^{(n)}(x)=u_{n+1}^{(n)}(x)$. Set 
$$
J_k=[a_{n+1}+\ldots+a_{n+k},a_n].
$$
Actually we shall prove that for any $x\in J_k$ we have
\begin{equation}\label{indass}
u_{\:n+k}^{(n)}(x)=u_{\:n}^{(n)}(x)=\frac{1}{a_{0}a_{1}\cdots a_{n}},
\end{equation}
where the last equality holds by Lemma~\ref{convol-basic} since $J_k\subset[0,a_n]=I_1$. This can be seen in Figure~\ref{overlay2} at the right half of the rectangular bump: where all three functions coincide. 
\begin{figure}[t]
    \centering
    \includegraphics[width=0.8\textwidth]{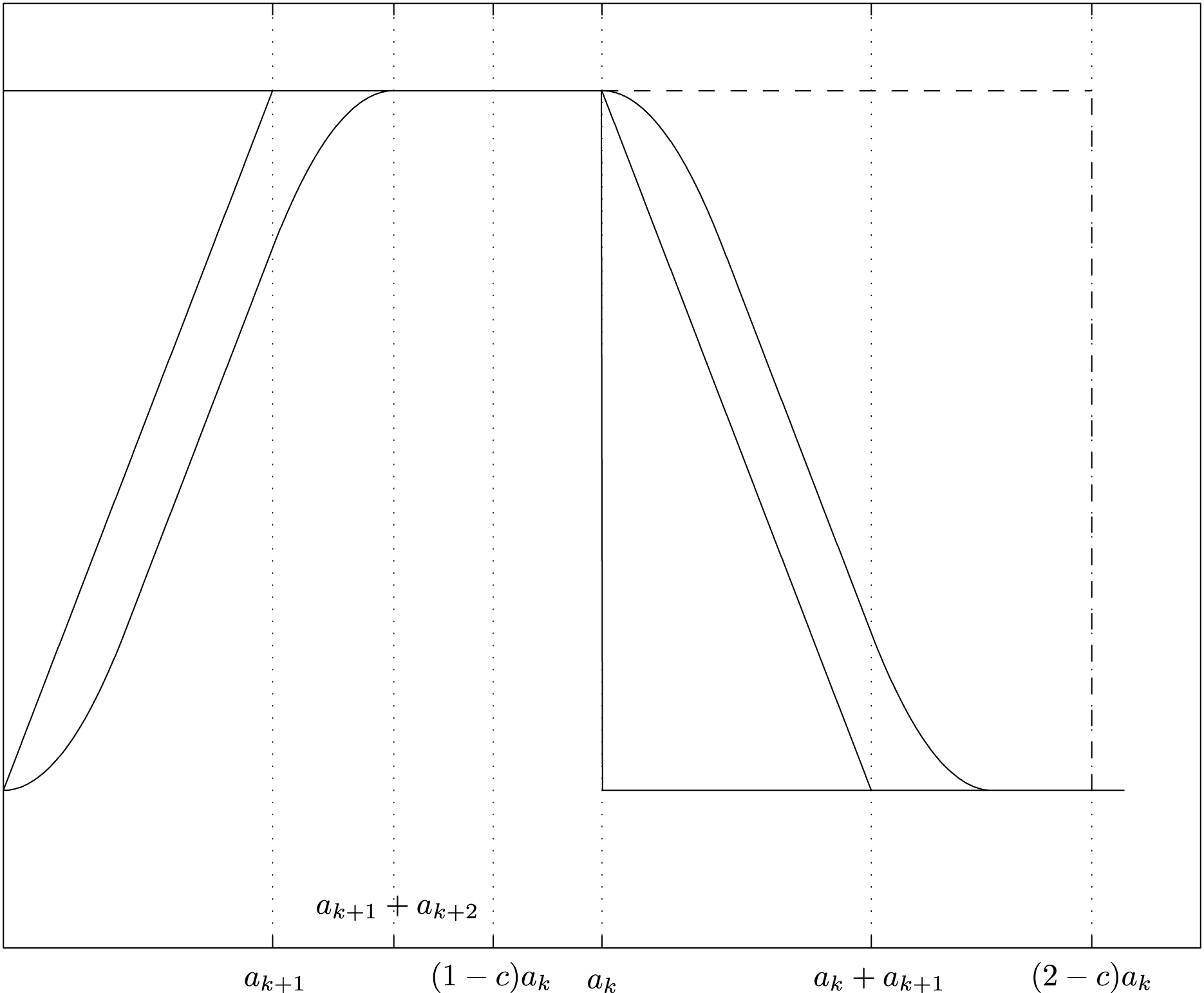}
    \caption{Overlay of a part the graphs of $u_{k}^{(k)}$, $u_{k+1}^{(k)}$ and $u_{k+2}^{(k)}$}
    \label{overlay2}
\end{figure}
The result is trivial for $k=0$. Assume \eqref{indass} holds for some arbitrary $k$ and pick an $x\in J_k$. Now, we have that
$$
u_{\:n+k+1}^{(n)}(x)=\frac{1}{a_{n+k+1}}\int_{0}^{a_{n+k+1}}u_{\:n+k}^{(n)}(x-t)\dd t.
$$
For each $t$ with $0\leq t\leq a_{n+k+1}$ we have that $a_n\geq x-t$ and also
$$
x-a_{n+k+1}\geq a_{n+1}+\ldots +a_{n+k}+a_{n+k+1}-a_{n+k+1}=a_{n+1}+\ldots+a_{n+k},
$$
so $x-t\in J_k$. Thus by the induction assumption, $u_{\:n+k}^{(n)}(x-t)$ is constant and equal to to the value of $u_{n}^{(n)}$ on $I_1$ on the whole region of integration. Therefore the first assertion follows, since we finally have
$$
u_{\:n+k+1}^{(n)}(x)=u_{\:n+k}^{(n)}(x)\frac{1}{a_{n+k+1}}\int_{0}^{a_{n+k+1}}\dd t=u_{\:n+k}^{(n)}(x)=u_{\:n}^{(n)}(x)=\frac{1}{a_0\cdots a_n}.
$$

For the second assertion just observe that $u^{(n)}$ is supported on $2^{n}$ intervals of length $a_{n}+a_{n+1}+\ldots$. On each of these intervals it will have the same shape, and we will estimate this by a rectangle of length $(2-c)a_n$, see the dotted rectangle in Figure~\ref{overlay2}. We get
\begin{multline*}
\nl u^{(n)}\nr_{p}^{p}=\int_{\R}\lvert u^{(n)}(x)\rvert^{p}\dd x = 2^{n}\int_{0}^{a_{n}+a_{n+1}+\ldots}\lvert u^{(n)}\rvert^{p}\dd x \\
\leq 2^{n}\int_{0}^{(2-c)a_n}\frac{1}{a_{0}^{p}\cdots a_{n}^{p}}\dd x = (2-c)\frac{2^{n}a_n}{a_{0}^{p}\cdots a_{n}^{p}}.
\end{multline*}
The last expression is exactly the $p$-th power of the $L^p$-norm of $u_{\:n}^{(n)}$.
\end{proof}

\subsection{Necessity of $\product{M}<\infty$}

In this section we shall investigate to what extent Theorem~\ref{completion-is-smooth} can be considered to be sharp. Under the assumption $\product{M}=\infty$ we will use the machinery developed in the previous section to construct functions $u_j\in\S$ which pointwise grow arbitrarily large, while remaining bounded in our quasi-norm, making any norm inequality of the type
$$
\nl u\nr_{\infty}\leq C\norm{u},\quad u\in\S
$$
impossible. This, in turn, shows that we cannot have embeddings of the kind we encountered in Theorem~\ref{completion-is-smooth} Our main result in this direction is presented in the following theorem.

\ifsmall\else
\newpage 
\fi

\begin{theorem}
\label{prod-infty-main}
Suppose $\product{M} = \infty.$ Assume either that
$$\liminf_{k\to\infty} \q^k\log M_k > 0$$
or that:
\begin{align}
&\text{$\log M_k$ is an increasing and convex sequence,}\tag{$i$}\label{convexity}\\
&\lim_{k\to\infty} \q^k\log M_k = 0, \tag{$ii$}\label{limq}\\
&\liminf_{k\to\infty} \frac{\log M_k}{P(k)} = \infty, \quad \text{for any polynomial $P$.}\tag{$iii$}\label{polygrowth}
\end{align}
Then there can be no constant $C$ such that
$$\nl f\nr_\infty \leq C\norm{f}, \quad f\in\S.$$
\end{theorem}

The first case concerns sequences $\M$ with at least exponential growth, i.e.~sequences such that eventually $\log M_k\geq C q^{-k}$ holds for some $C>0$.

With regards to the second case, condition \eqref{polygrowth} might seem out of place. However, for $\product{M}=\infty$ to hold it is necessary that
$$
\limsup_{k\to\infty} \frac{\log M_k}{P(k)} = \infty
$$
for any polynomial $P$. Hence condition \eqref{polygrowth} is in a sense also a regularity on the sequence $\M$, removing the possibility that some subsequences are bounded by some polynomial.
Together with condition \eqref{limq}, which says that $\log M_k$ cannot grow as quick as $\q^{-k}$, this somehow tries to pinpoint the region between exponential and polynomial growth.

In all, there is not much wiggle-room for a sequence $\M$ with $\product{M}=\infty$ to not satisfy the conditions, except for oscillations around $\exp(\q^{-k})$.

We begin by proving that the only case we need to consider in Theorem~\ref{prod-infty-main} is the second one, since sequences growing at least as quick as $\q^{-k}$ can easily be handled. We will do this by constructing a smaller weight sequence, adhering to the second case of Theorem~\ref{prod-infty-main}. The norm inequality \eqref{norm-ineq} will ensure that the theorem holds for the original sequence.
\begin{lemma}\label{too-quick}
If the condition
$$\liminf_{k\to\infty} \q^k \log M_k > 0$$
is met, there exists a smaller sequence $\widetilde{\M}=\seq{\tM}$ such that $\mu_{\widetilde{\M}}=\infty$ and satisfies the conditions \eqref{convexity}-\eqref{polygrowth} of Theorem~\ref{prod-infty-main}.
\end{lemma}

\begin{proof}
By assumption there exists a constant $C$ and an integer $k_0$ such that
$$\q^k\log M_k \geq C > \frac{1}{k}, \quad k>k_0.$$
Therefore define $\tM_k$ for $k>k_0$ by
$$\log\tM_k = \frac{\q^{-k}}{k}$$
and so $M_k\geq \tM_k$ for $k>k_0$.
Note also that we get
$$\liminf_{k\to\infty} \frac{\log\tM_k}{P(k)} = \liminf_{k\to\infty} \frac{\q^{-k}}{kP(k)} = \infty$$
where $P$ is any polynomial.
Then $\log\tM_k$ is both convex and increasing if $k$ is large enough.
Indeed, if we set $g(x)=\frac{\q^{-x}}{x}$ then $g''$ is always positive and $g'(x)>0$ if $k\geq\frac{1}{-\log\q}=:k_1$.

By construction
$$\q^{k}\log\tM_k = \frac{1}{k}$$
so that
$$\lim_{k\to\infty} \q^{k}\log\tM_k = 0 \and \sum_{k\geq N} \q^{k}\log\tM_k = \infty, \quad N=\max(k_0,k_1).$$

Now we want to define $\log\tM_k$ for $0\leq k < N$ so that the whole sequence becomes positive, increasing and convex.
This can clearly be achieved by setting $\log\tM_k$ to be the minimum of $\log M_k$ for $0\leq k \leq N$ and then redefine finitely many $\log\tM_k$ after $k\geq N$ so that the whole sequence becomes convex.
\end{proof}

Now we turn to studying the remaining case of Theorem~\ref{prod-infty-main}.
As mentioned we want to use the machinery in the previous section to construct a function $u$ belonging to $\S$.
We can then use Proposition~\ref{convol-est} and try to choose the sequence $\seq{a}$ so that the expression
$$\nl u_n^{(n)}\nr_p \leq (2-c)^{1/p}\frac{2^{k/p}a_k^{1/p}}{a_0 a_1 \cdots a_k}$$
is bounded in a suitable way, to force $u$ to have finite $\M$-norm, and to satisfy
$$\limsup_{k\to\infty} \left(\frac{1}{a_0 a_1 \cdots a_k}\right)^{\q^k}\leq 1.$$
It will be easier to study the logarithmized versions of these requirements.
To this end we set $\alpha_i = -\log a_i$ so that these expressions read:
$$\frac{1}{p}\log(2-c) + \frac{k}{p}\log 2 - \frac{1}{p}\alpha_k + \sum_{i=0}^k \alpha_i$$
and
$$\limsup_{k\to\infty}  \q^k \sum_{i=0}^k \alpha_i \leq 0.$$

To further simplify the construction of our sequences we set
$$a_{ij}=g_{ij} e^{-\frac{1}{1-c}i}$$
for a sequence $g=(g_{ij})_{i,j\geq 0}$ with $0<g_{ij}\leq 1$ and $g_{(i+1)j}\leq g_{ij}$.
We do this because this makes (\ref{separation}) hold automatically.
Indeed, for any $j$:
\ifsmall
\begin{align*}
\sum_{i>k} a_{ij} \leq g_{kj} \sum_{i>k} e^{-\frac{1}{1-c}i} &= g_{kj}e^{-\frac{1}{1-c}k}\sum_{i-k>0}e^{-\frac{1}{1-c}(i-k)}\\
 & =a_{kj}\frac{e^{-\frac{1}{1-c}}}{1-e^{-\frac{1}{1-c}}} \leq (1-c) a_{kj}.
\end{align*}
\else
$$\sum_{i>k} a_{ij} \leq g_{kj} \sum_{i>k} e^{-\frac{1}{1-c}i} = g_{kj}e^{-\frac{1}{1-c}k}\sum_{i-k>0}e^{-\frac{1}{1-c}(i-k)}=a_{kj}\frac{e^{-\frac{1}{1-c}}}{1-e^{-\frac{1}{1-c}}} \leq (1-c) a_{kj}.$$
\fi

If we can achieve that $g_{ij}\to 0$ as $j\to\infty$ for any fixed $i$ then since $a_{ij}\leq e^{-\frac{1}{1-c}i}$ we get by the dominated convergence theorem that
$$\sum_{i=0}^\infty a_{ij} \to 0.$$
Hence the support of the corresponding functions $u_j$, constructed by $a_j=(a_{ij})_{i\geq0}$, will tend to zero as $j\to\infty$.

The final form of the requirements can now be rephrased to incorporate the latest simplification.
If we set $\gamma_{ij}=-\log g_{ij}$ the first expression becomes
$$\frac{1}{p}\log(2-c) + \frac{k}{p}\log 2 - \frac{1}{p}\gamma_{kj} - \frac{k}{p(1-c)} + \sum_{i=0}^k \gamma_{kj} + \frac{k(k+1)}{2(1-c)}$$
or more neatly stated as
\begin{equation}
\label{finite-norm}
P(k) - \frac{1}{p}\gamma_{kj} + \sum_{i=0}^k \gamma_{kj}
\end{equation}
where $P$ is the second order polynomial
$$P(k)=\frac{1}{p}\log(2-c) + \frac{k}{p}\log 2 + \frac{k(k+1)}{2(1-c)}.$$
The second is satisfied if we have simply that
\begin{equation}
\label{log-liminf}
\limsup_{k\to\infty} \q^k \sum_{i=0}^k \gamma_{ij}\leq 0.
\end{equation}
since
$$\limsup_{k\to\infty} \q^k \sum_{i=0}^k \alpha_i = \limsup_{k\to\infty} \q^k \sum_{i=0}^k \gamma_{ij} + \lim_{k\to\infty}\q^k\frac{k(k+1)}{2(1-c)} \leq 0.$$

Now that we know which requirements we want our sequences to meet we show that this is possible in the setting of Theorem~\ref{prod-infty-main}.
\begin{lemma}
Suppose $\product{M} = \infty$ and that conditions \eqref{convexity} and \eqref{limq} of Theorem~\ref{prod-infty-main} are fulfilled.
Then there exists a positive double sequence $(\gamma_{ij})_{i,j\geq 0}$ such that for any $j$ the sequence $\gamma_j=(\gamma_{ij})_{i\geq 0}$ satisfy (\ref{log-liminf}) and
\begin{equation}
\label{half-finite-norm}
- \frac{1}{p}\gamma_{kj} + \sum_{i=0}^k \gamma_{kj} \leq \q \log M_k.
\end{equation}
Furthermore $\gamma_{ij}\to \infty$ as $j\to\infty$ for any fixed $i$ and $\gamma_{ij} \geq \gamma_{(i-1)j}$.
\end{lemma}

\begin{proof}
Instead of defining $\gamma_{ij}$ directly we consider an auxiliary function
$$h(k,j)=\q^k \sum_{i=0}^k \gamma_{ij},$$
so that
$$\sum_{i=0}^k \gamma_{ij} = \frac{h(k,j)}{\q^k}$$
and let $h$ define $\gamma_{kj}$ by the two equivalent expressions, which follows from the telescoping nature of the definition of $h$ as a partial sum of the $\gamma_{ij}$:
$$\gamma_{kj} = \frac{h(k,j)-\q h(k-1,j)}{\q^k} = \frac{h(k,j)-h(k-1,j)}{\q^k} + p\frac{h(k-1,j)}{\q^k}.$$

To choose a relevant function $h$ we take the logarithm of the equation $\product{M}=\infty$:
$$p\sum_{i=0}^\infty \q^i\log M_i = \infty$$
(where we have multiplied this by $\q$ to ease notation) and set
$$h(k,j)=p\sum_{i=k+1}^{j+k} \q^i\log M_i.$$
The idea here is that since the sum diverges we get that $h(k,j)\to\infty$ as $j\to\infty$ for fixed $k$.

We start with verifying (\ref{log-liminf}) by estimating
\begin{align*}
h(k,j) &= p\sum_{i=k+1}^{j+k} \q^i\log M_i \leq p\q^{j+k}\log M_{j+k} \sum_{i=k+1}^{j+k} \q^{i-(j+k)}
\end{align*}
where the sum does not depend on $k$ since
\begin{align*}
\sum_{i=k+1}^{j+k} \q^{i-(j+k)} &= \{i=k+1+i'\} = \q^{1-j}\sum_{i'=0}^{j-1} \q^{i'}\\
& = \q^{1-j} \frac{1-\q^j}{1-\q} = \frac{\q}{p}\big(\q^{-j}-1\big).
\end{align*}
Therefore we can use the assumption $\lim_{k\to\infty} \q^k\log M_k=0$ to assert that
$$\limsup_{k\to\infty} \q^k \sum_{i=0}^k \gamma_{ij} = \limsup_{k\to\infty} h(k,j)= 0$$
and this is (\ref{log-liminf}).

To verify that we can satisfy (\ref{half-finite-norm}) we calculate
$$- \frac{1}{p}\gamma_{kj} + \sum_{i=0}^k \gamma_{kj} = - \frac{h(k,j)-\q h(k-1,j)}{p\q^k}+\frac{h(k,j)}{\q^k}=\frac{h(k-1,j) - h(k,j)}{p\q^{k-1}}.$$
Hence we are interested in the expression
\begin{align*}
h(k-1,j) - h(k,j) & = p\sum_{i=k}^{j+k-1}\q^i\log M_i - p\sum_{i=k+1}^{j+k} \q^i\log M_i\\
&= p\q^k\log M_k - p\q^{j+k}\log M_{j+k} \leq p\q^k\log M_k
\end{align*}
and therefore
$$- \frac{1}{p}\gamma_{kj} + \sum_{i=0}^k \gamma_{kj} \leq \q\log M_k$$
which is (\ref{half-finite-norm}).

To check the properties of $\gamma_{kj}$ we need to calculate
\begin{align*}
h(k,j)-qh(k-1,j) &= p\sum_{i=k+1}^{j+k}\q^i\log M_i - \underbrace{p\sum_{i=k}^{j+k-1}\q^{i+1}\log M_i}_{\text{set $i'=i+1$}}\\
&= p\sum_{i=k+1}^{j+k}\q^i\log\frac{M_i}{M_{i-1}}.
\end{align*}
Therefore since $M_i\geq M_{i-1}$ we get that $h(k,j)-qh(k-1,j)\geq 0$ and then $\gamma_{kj}\geq 0$.

The growth of $\gamma_{kj}$ follows if we note that for fixed $k$ we have
$$h(k-1,j) - h(k,j)=p\q^k\log M_k - p\q^{j+k}\log M_{j+k}$$
so that
$$\big|h(k-1,j) - h(k,j)\big| \leq p\q^k\log M_k + p\q^{j+k}\log M_{j+k}$$
which is bounded.
Indeed, the last term tends to zero as $j$ tends to infinity by assumption.
Therefore
$$\gamma_{kj} = \frac{h(k,j)-h(k-1,j)}{\q^k}  + p\frac{h(k-1,j)}{\q^k} \to \infty, \quad j\to\infty.$$

The last thing we need to check is that $\gamma_{kj} \geq \gamma_{(k-1)j}$:
\begin{align*}
\gamma_{kj}-\gamma_{(k-1)j} &=\frac{h(k,j)-qh(k-1,j)}{\q^k}-\frac{h(k-1,j)-\q h(k-2,j)}{\q^{k-1}}\\
&= p\sum_{i=k+1}^{j+k}\q^{i-k}\log\frac{M_i}{M_{i-1}} - p\underbrace{\sum_{i=k}^{j+k-1}\q^{i-k+1}\log\frac{M_i}{M_{i-1}}}_{\text{set $i'=i+1$}}\\
&= p\sum_{i=k+1}^{j+k}\q^{i-k}\left(\log\frac{M_i}{M_{i-1}} - \log\frac{M_{i-1}}{M_{i-2}}\right).
\end{align*}
Where the last inequality follows by the convexity of $\log M_k$; that is,
\begin{equation*}
\log M_i + \log M_{i-2} - 2\log M_{i-1}\geq 0 \geq 0,\quad i\geq 2. \qedhere
\end{equation*}
\end{proof}

Now we are finally ready to prove this section's main theorem.

\begin{proof}[Proof of Theorem \ref{prod-infty-main}]
Without loss of generality we can assume that the situation is as described by the conditions \eqref{convexity}-\eqref{polygrowth}. Indeed, if not, by Lemma~\ref{too-quick} we get a smaller class which falls under these assumptions, and by \eqref{norm-ineq} we cannot have the norm inequality under consideration for the bigger class either.

Under these conditions on $\M$ we can apply this lemma and construct a sequence of functions $u_j$ with vanishing supports as $j\to\infty$, having integral one and
$$\nl u_j^{(k)}\nr_p \leq e^{P(k)} M_k^{\q}$$
so that
$$\norm{u_j} \leq \sup_{k\geq 0} \frac{e^{P(k)}}{M_k^p}.$$
By the condition \eqref{polygrowth} we see that
$$\norm{u_j}\leq C$$
for some constant $C$ depending only on the sequence $M_k$ and therefore they are uniformly bounded in $\S$ and their supremum norms tends to infinity. This last part deserves a comment. Condition~\eqref{polygrowth} implies that $\log M_k\geq P(k)$ eventually, so for large $k$ we have $M_k\geq e^{P(k)}$. This clearly ensures boundedness of the quotient.
\end{proof}

\subsection{Necessity of a bound on $\limsup \nl f^{(n)}\nr_{\infty}^{\q^{n}}$}

It is natural to ask if the requirement \eqref{liminf} is really necessary. After all, it appeared in a calculation when we tried to estimate the supremum norm of a function from above by a product involving the $L^p$-norms of its derivatives --- we could get a bound, but only after assuming that the supremum norms don't grow too quickly with the order of the derivative.

We will not investigate this question with any great resolution, but only say what happens if we drop it altogether. In some sense, this theorem is also very close to a direct generalization of Peetre's main result concerning $W^{p,k}$ in \cite{Peetre} to the case $k=\infty$.

Now for the main result in this direction.
\begin{theorem}\label{nlmain}
With the notation above; there exists a canonical, continuous embedding 
$$
L^{p}(\R) \hookrightarrow \Ccomp.
$$
Explicitly one can map $f\in L^{p}$ to a Cauchy sequence $(f_i)$ in $\C$ such that $f_i\to f$ in $L^p$ and for each derivative we have \mbox{$f_i^{(n)}\to0$} in $L^{p}$.
\end{theorem}

\begin{remark}
In this theorem, and in fact also in Theorem~\ref{completion-is-smooth}, one could equally well have chosen the quasi-norm
$$
\nl f\nr=\left(\sum_{k\geq1}\frac{\nl f^{(k)}\nr_{p}^{p}}{M_k^{p}}\right)^{1/p}.
$$
The interested reader will be able to fill in the details.
\end{remark}

We need the following sequence of results, which constructs mollifiers in $\C$ which we can then use to approximate step functions, which in turn approximate $L^{p}$-functions.

\begin{lemma}\label{nlmollifier} Let $\epsilon_1$ and $\epsilon_2$ be arbitrary positive numbers. Then we can find a non-zero function $v\in \C$ such that
\begin{enumerate}
\item for any $k\geq0$ we have
\begin{equation}\label{moll-deriv}
\nl v^{(k)}\nr_p<\epsilon_1 M_{k+1},\quad k\geq0,
\end{equation}
\item $\supp v\subset [0,\epsilon_2]$,
\item $v$ has integral equal to one:
$$
\int_{\R} v(t)\dd t=1.
$$
\end{enumerate}
\end{lemma}

\begin{remark}
We can think of this as saying that we can find $v$ in another class $\widetilde{\C}$ defined with respect to the shifted sequence $\mathcal{N}=\seq{N}$ where $N_k=M_{k+1}$, such that $\nl v \nr_{\widetilde{\C}}<\epsilon_1$. The control of the $p$-norms of $v^{(k)}$ can then be summarized as $\nl v\nr_{\mathcal{N}}\leq\epsilon_1$.
\end{remark}
We will do this by using the infinite convolutions previously discussed.

\begin{proof}
If $\seq{a}$ is a sequence of positive numbers and $u_n$ is given by \eqref{unconstruct}, we let $u$ denote its limit as $n\to \infty$.

To be able to use the machinery for the infinite convolutions developed in the preceding sections, recall that we have to fulfill the requirement \eqref{separation} for some number $0<c<1$, i.e.
\begin{equation}
\sum_{j>k}a_j\leq (1-c)a_k,\quad k\geq0.
\end{equation}
Then $\nl u^{(n+1)}\nr_{p}\leq \epsilon_1 M_{n+1}$ can be ensured by using the estimate from Proposition \eqref{convol-est} and requiring that
$$
(2-c)^{1/p}\frac{2^{n/p}a_n^{1/p}}{a_0\cdot \ldots \cdot a_n}\leq \epsilon_1 M_{n+1}, \quad n\geq0.
$$
Solving for $a_n$ we find that this is equivalent to
\begin{equation}\label{ancondition}
a_{n}^{\frac{1-p}{p}}\leq \epsilon_1 M_{n+1}\frac{a_0\cdots a_{n-1}}{(2-c)^{1/p}2^{n/p}}, \quad n\geq1.
\end{equation}
Observe that if $M_{n+1}$ and $a_j, j=0,1,\ldots, n-1$ are given, then this can always be ensured to hold by choosing $a_n$ small enough.

Now suppose we have a sequence that satisfies \eqref{separation}, but violates \eqref{ancondition} for some $n_0$. The idea is then to redefine $a_{n_0}$ so that \eqref{ancondition} holds and then redefine the tail $(a_n)_{n\geq n_{0}}$ so that it is at least geometrically decreasing, ensuring \eqref{separation}. Explicitly we let $a_0$ be arbitrarily chosen with the only requirement that $a_0\leq (\epsilon_1 M_1)^{p/(1-p)}$. We then define the remaining numbers recursively by
$$
a_n=\min\left\{a_0r^{n}, \ldots, a_kr^{n-k}, \ldots, a_{n-1}r, \left(\epsilon_1 M_{n+1} \frac{a_0\cdots a_{n-1}}{(2-c)^{1/p}2^{n/p}}\right)^{p/(1-p)}\right\},
$$
where $r<1$ is defined by the requirement $\sum_{j>0} r^j=1-c$.

That $(a_n)_{n\geq0}$ satisfies \eqref{ancondition} is trivially true. The same holds for \eqref{separation}. Indeed, since by the definition of $a_j$ as a minimum we get $a_j\leq a_{j-k}r^{j-k}$, and therefore
$$
\sum_{j>k}a_j\leq a_k\sum_{j-k>0}r^{j-k}=(1-c)a_k.
$$

Lastly, the measure of the support of $u$ is
$$
m(\supp u) = \sum_{j\geq0}a_j = a_0+\sum_{j>0}a_j< (2-c)a_0.
$$
Since $a_0$ was arbitrarily chosen (up to an upper bound) this can certainly assumed to be less than $\epsilon_2$. By construction, all infinite convolutions of this type have integral $1$.
\end{proof}

On our way to approximate all $L^p$-functions, we now use these mollifiers to approximate characteristic functions of arbitrary intervals.

\begin{lemma}\label{nlstep} 
For any interval $I=[a,b]$ and any $\epsilon$ and $\eta>0$ we can find $u\in\C$ such that
$$
\nl u-\chi_I\nr_{p}<\epsilon \and \nl u^{(k)}\nr_{p}\leq \eta M_k,\quad k\geq1.
$$
\end{lemma}

\begin{proof}
Consider $u=\chi_I\ast v$ where $v$ is the function constructed in the previous lemma for some $\epsilon_1$ and $\epsilon_2$ which we will describe below.

Note that
$$u^{(n+1)}(x) = \left(\delta_a-\delta_b\right)\ast v^{(n)}(x) = v^{(n)}(x-a) - v^{(n)}(x-b)$$
so that if we control the support of $v$ by choosing $0<\epsilon_2<\frac{b-a}{2}$ we get no overlap and therefore
$$\nl u^{(n+1)}\nr_p = 2^{1/p}\nl v^{(n)}\nr_p.$$
Therefore, by applying \eqref{moll-deriv} to the above expression, we can control the size of the derivatives of $u$ by
$$\frac{\nl u^{(n+1)}\nr_p}{M_{n+1}} \leq 2^{1/p} \epsilon_1.$$

To show that our $u$ approximate $\chi_I$ in $L^p$ note that $u$ and $\chi_I$ coincide unless $x\not\in(a,a+\epsilon_2)\cup(b,b+\epsilon_2)$. The asymmetry in this expression is due to the fact that our mollifier $v$ has support $\supp v=[0,\epsilon_2]$. Furthermore $|u|\leq 1$ which simply implies that $|u-\chi_I|\leq 2$ and so we get the desired control
\begin{equation*}\nl u-\chi_I\nr_p \leq \left(2\cdot2^p\cdot\epsilon_2\right)^{1/p}.\qedhere\end{equation*}
\end{proof}

The next step is naturally to do the same for step functions.

\begin{lemma}
For any step function $f$, that is, a finite linear combination of characteristic functions of disjoint intervals:
$$f=\sum_k a_k \chi_{I_k},$$
there exists a Cauchy sequence $\{f_i\}$ in $\C$ such that
$$\nl f-f_i\nr_p \to 0 \and \frac{\nl f_i^{(n)}\nr_p}{M_n} \to 0, \quad i\to\infty.$$
In particular this implies that
$$\norm{f_i} \to \frac{\nl f\nr_p}{M_0}, \quad i\to\infty.$$
\end{lemma}

Before proceeding to the proof we would like to remind the reader of the fact that a triangle inequality \eqref{triangle} still holds, if we raise all quasi-norms to the power $p$, even though $0<p<1$.

\begin{proof}
For any sequence $\eta_i$ tending to zero we approximate each $\chi_{I_k}$ with $u_{kj}$ by the previous lemma with
$$\nl \chi_{I_k} - u_{ki}\nr_p \leq \eta_i$$
and
$$\nl u_{ki}^{(n)} \nr_p \leq M_n\eta_i, \quad n\geq 1.$$
Using these we set
$$f_i = \sum_k a_k u_{ki}.$$
This function approximate $f$ in $L^p$ by
$$\nl f - f_i \nr_p^p \leq \sum_k |a_k| \nl \chi_{I_k} - u_{ki}\nr_p^p \leq \eta_i^p \left(\sum_k |a_k|\right)$$
and have derivatives satisfying
$$\frac{\nl f_i^{(n)} \nr_p^p}{M_n^p} \leq \eta_i^p \left(\sum_k |a_k|\right), \quad n\geq 1.$$
This clearly implies that
$$\lim_{i\to\infty} \norm{f_i} = \frac{\nl f\nr_p}{M_0}.$$

Finally to see that this sequence is a Cauchy sequence in $\C$ we calculate:
$$
\norm{f_i-f_j}^p = \max\left\{ \frac{\nl f_i - f_j\nr_p^p}{M_0^p}, \sup_{n\geq 1} \frac{\nl f_i^{(n)} - f_j^{(n)}\nr_p^p}{M_n^p} \right\}.
$$
We can estimate the first expression by
\begin{equation}\label{est-triang-1}
\frac{\nl f_i - f_j\nr_p^p}{M_0^p} \leq \frac{\nl f_i - f\nr_p^p}{M_0^p} + \frac{\nl f - f_j\nr_p^p}{M_0^p}
\end{equation}
and the second by
\begin{equation}\label{est-triang-2}
\sup_{n\geq 1} \frac{\nl f_i^{(n)} - f_j^{(n)}\nr_p^p}{M_n^p}\leq  \sup_{n\geq 1} \frac{\nl f_i^{(n)}\nr_p^p}{M_n^p} +  \sup_{n\geq 1} \frac{\nl f_j^{(n)}\nr_p^p}{M_n^p}.
\end{equation}
Thus by previous estimates, both \eqref{est-triang-1} and \eqref{est-triang-2} can be estimated in terms of the numbers $\eta_i$ and the coefficients $a_i$, and we arrive at
\begin{equation*}
\norm{f_i-f_j}^p \leq \left(\eta_i^p+\eta_j^p\right) \left(\sum_k |a_k|\right). \qedhere
\end{equation*}
\end{proof}

The reader who has followed us this far, through the sequence of three technical lemmas, will probably be delighted to see that it pays off. Now follows the rather succinct proof of the main theorem of this section.

\begin{proof}[Proof of Theorem \ref{nlmain}]
We define $\beta$ to be the map from $L^p$ into $\Ccomp$ which maps $f$ to any Cauchy sequence $\{f_i\}$ in $\C$ with the properties that
$$\nl f-f_i\nr_p \to 0,\quad \frac{\nl f_i^{(n)}\nr_p}{M_n} \to 0, \quad i\to\infty.$$
These two properties directly imply the continuity since we have
$$\norm{\beta(f)}:=\lim_{i\to\infty} \norm{f_i} = \frac{\nl f\nr_p}{M_0}.$$
That such a sequence exists for any $f$ is clear since the step functions are dense in $L^p$ and the previous lemma.
Hence we only need to check that this map is both well-defined and injective.

We begin with the well-definedness and to this end let $\{g_i\}$ be another candidate Cauchy sequence in $\C$, to which $f$ could just as well have been mapped.
We want to see that they are equivalent in the completion.
Therefore we argue as in the previous lemma and see that $\norm{f_i-g_i}^p \to 0$ as $i\to\infty$, since all the quantities in the right hand sides of \eqref{est-triang-1} and \eqref{est-triang-2} tend to zero by the assumptions on the sequences.

That $\beta$ is injective is quite easy if we just remember that $\beta(f)=0$ means that a Cauchy sequence $\{f_i\}$ in $\C$ satisfying the above conditions is equivalent to the zero Cauchy sequence.
That is,
$$\norm{f_i}\to 0, \quad i\to\infty,$$
but this implies, by Lemma~\ref{l:ssublp} that
$$\nl f_i \nr_p \leq M_0 \norm{f_i}\to 0, \quad i\to\infty,$$
and therefore $f=0$ in $L^p$.
\end{proof}

\section{Some extension to the unit circle}
We have so far been working exclusively on the line.
However there is not much that does not immediately carry over to the unit circle $\T$.
One thing that was utilized in the proof of Theorem~\ref{completion-is-smooth} that works on $\R$ but not on $\T$ is that $L^p$-functions must attain arbitrarily small values somewhere.
This is not the case on a finite measure space.
For this reason we only get control of the oscillation from the integral mean of a function.
This analogous result, however, turns out to be enough. Here we show how to fill in the details.

All other results are local ones and we expect the rest of the theory to carry over without change.

We shall allow ourselves to keep denoting the quasi-norm
$$
\norm{f}=\sup_{k\geq0}\frac{\nl f^{(k)}\nr_{p}}{M_k}
$$
despite the face that the $p$-norms are now taken as integrals over $\T$.
We will be quite unconventional and choose to not renormalized the measure, so the circle has mass $2\pi$.
This will become apparent in the proof below.

We will consider the class $\ST$ defined, almost exactly as $\S$, by
$$
\ST=\left\{ f\in C^{\infty}(\T): \norm{f}<\infty \and \limsup_{k\to\infty} \nl f^{(k)}\nr_{\infty}^{\q^{k}}\leq 1\right\}.
$$
Note that all $f\in C^{\infty}(\T)$ are periodic in all derivatives.

For this class we will prove analogous result of Theorem~\ref{completion-is-smooth}.
\begin{theorem}\label{completion-is-smooth-T}
Assume that the sequence $\M$ satisfies $\product{M}<\infty$. Then the completion of $\ST$ in the $\ST$-norm can be canonically and continuously embedded into $C^{\infty}(\T)$.
\end{theorem}

This will follow from suitably altered equivalents of the results of Section~2.

\begin{proposition}\label{sup-est-T}
Let $f\in \ST$. Then
$$
\nl f-f_{\T}\nr_{\infty}\leq \product{M}\normT{f},
$$
where $u_{\T}$ denotes the mean
$$
f_{\T}=\frac{1}{2\pi}\int_{\T}f(x)\dd x.
$$
\end{proposition}
\begin{proof}
We have that
$$
\lvert f(x)-f_{\T}\rvert=\left\lvert f(x)-\frac{1}{2\pi}\int_{\T}f(t)\dd t\right\rvert=\left\lvert \frac{1}{2\pi}\int_{\T}(f(x)-f(t))\dd t\right\rvert.
$$
But
$$
\lvert f(x)-f(t)\rvert\leq\int_{t}^{x}\lvert f'(s)\rvert \dd s\leq \int_{\T}\lvert f'(s)\rvert \dd s\quad (x,t\in\T).
$$
Putting these two expressions together yields
$$
\lvert f(x)-f_{\T}\rvert \leq \frac{1}{2\pi}\int_{\T}\int_{\T}\lvert f'(s)\rvert\dd s\dd t=\int_{\T}\lvert f'(s)\rvert\dd s\quad (x\in\T).
$$
Now take supremum over $x$, and write $\lvert f'\rvert=\lvert f'\rvert^{p}\lvert f'\rvert^{1-p}$ to obtain
$$
\nl f-f_{\T}\nr_{\infty}\leq\int_{\T}\lvert f'(t)\rvert\dd t\leq \nl f'\nr_{\infty}^{1-p}\nl f'\nr_{p}^{p}.
$$
Observe that $f^{(n)}_{\T}=0$ for all $n\geq1$: indeed all derivatives $f^{(k)}, k\geq0$ are periodic, and we can evaluate the integral defining $f^{(n)}_{\T}$ as
$$
f^{(n)}_{\T}=\int_0^{2\pi} f^{(n)}(t) \dd t = f^{(n-1)}(2\pi)-f^{(n-1)}(0)=0.
$$
Thus $\nl f'\nr_{\infty}^{1-p}\nl f'\nr_{p}^{p}=\nl f'-f'_{\T}\nr_{\infty}^{1-p}\nl f'\nr_{p}^{p}$.
We are now in a position to iterate this procedure (replacing first $f-f_{\T}$ by $f'-f'_{\T}$, etc.) to obtain
$$
\nl f-f_{\T}\nr_{\infty}\leq \nl f^{(n)}-f^{(n)}_{\T}\nr_{\infty}^{\q^{n}}\times \prod_{k=1}^{n}\nl f^{(k)}\nr_{p}^{p\q^{k-1}}.
$$
Thus we are left in the situation of the proof of Proposition \ref{sup-est}, which we know how to handle.
\end{proof}

As before we can rephrase this result so that it applies to the derivatives.

\begin{corollary}\label{derivative-sup-est-T}
Let $f\in\ST$. Then
$$
\nl f^{(i)}\nr_{\infty}\leq\product{M}^{1/\q^{i}}\normT{f} \quad (i=1,2,\ldots).
$$
\end{corollary}

\begin{proof}
As before assume that $\normT{f}=1$ and use the same argument as in Proposition \ref{sup-est} but starting with $f^{(i)}$ instead of $f$ (remember that $f_\T^{(i)}=0$).
After the iteration we then have
$$\nl f^{(i)}\nr_{\infty}\leq \nl f^{(i+n)}\nr_{\infty}^{\q^{n}}\prod_{k=1}^{n}\nl f^{(i+k)}\nr_{p}^{p\q^{k-1}}.$$
Here we manipulate the product exactly as in Corollary \ref{derivative-sup-est} and get
$$\nl f^{(i)}\nr_{\infty}\leq\product{M}^{1/\q^{i}}$$
where we have let $n$ tend to infinity.
Hence the desired inequality follows.
\end{proof}

\begin{proof}[Proof of Theorem~\ref{completion-is-smooth-T}]
First observe that the sequence of derivatives is a Cauchy sequence in the supremum norm.
Indeed, by Corollary \ref{derivative-sup-est-T} we have that for $i\geq1$:
$$
\nl f_{j}^{(i)}-f_{k}^{(i)}\nr_{\infty}\leq\mu^{1/\q^{i}}\normT{f_j-f_k}\to0, \quad j,k\to\infty.
$$
Denote the oscillation of a function $h$ on $\T$ by $\operatorname{osc}(h)$. Then since $\{f'_j\}$ is Cauchy in sup-norm and $p$-norm (the latter norm is controlled by $\norm{\cdot}$-norm) we have
$$
\operatorname{osc}(f_j-f_k)=\sup_{x,y\in\T}\lvert (f_j-f_k)(x)-(f_j-f_k)(y)\rvert\leq \nl f'_j - f'_k\nr_{\infty}^{1-p}\nl f'_j-f'_k\nr_{p}^{p}
$$
by a now familiar argument. The $L^{p}$-norm clearly tends to $0$ as $j,k\to \infty$, and by the above calculation the same holds for the supremum norm expression.

For simplicity of notation, set $h_{j,k}=f_j-f_k$. For any $x\in\T$ we find that
\begin{multline*}
\lvert h_{j,k}(x)\rvert=\frac{1}{2\pi}\nl h_{j,k}(x)-h_{j,k}(y)+h_{j,k}(y)\nr_{p} \\ \leq \frac{K}{2\pi} \left[\left(\int_{\T}\lvert h_{j,k}(x)-h_{j,k}(y)\rvert^{p} \dd y\right)^{1/p}+\nl h_{j,k}\nr_{p} \right],
\end{multline*}
where $K$ is the constant required for the quasi-triangle inequality to hold. Taking supremum over all $x\in\T$ we arrive at
$$
\nl h_{j,k}\nr_{\infty}\leq \frac{K}{2\pi} \left(\operatorname{osc}(h_{j,k})+\nl h_{j,k}\nr_{p}\right)\to0,\quad j,k\to\infty,
$$
establishing the assertion of the theorem.
\end{proof}

\section{Conclusions and conjectures}

As discussed in the introduction, our starting point in these investigations was the observation made by Peetre \cite{Peetre}, that Sobolev spaces for $0<p<1$ behaves pathologically. The space $W^{k,p}$, defined as an abstract completion of $C^{\infty}$ with respect to the usual Sobolev quasi-norm is actually topologically isomorphic to $L^p$. We study the case when $k=\infty$, with a weighted and slightly different norm which is more inspired by expressions encountered in the study of Carleman classes on the real line;
$$
\norm{f}=\sup_{k\geq0}\frac{\nl f\nr_{p}}{M_k}
$$
where $\M:=\seq{M}$ is a weight sequence.

If one considers the completion $\Ccomp$ of the whole of $C^{\infty}$ with respect to this norm, the situation turns out to be much alike the one Peetre encountered for finite $k$. Indeed, a crucial ingredient in proving the isomorphism $W^{k,p}\cong L^{p}$ is the existence of a canonical, continuous injection $\beta: L^{p}\to W^{k,p}$ and we prove the existence of such a mapping in Theorem~\ref{nlmain}.

If one considers completions of another subclass, which we denote by $\S$, something completely different happens. A first result is that when taking the completion with respect to the quasi-norm $\norm{\cdot}$ one ends up with a subspace of $C^{\infty}$, provided that $\product{M} < \infty$ where $\product{M}$ is an expression describing the growth of $\M$. This result is sharp up to some regularity assumptions on $\M$, in the sense that when $\product{M}=\infty$ no such embedding is possible.

We expect that this is not all there is to it. We believe strongly that the following conjectures holds true.
\begin{conjecture}
For the space $\Ccomp$ we have
\begin{equation*}\label{uncoupling}
\Ccomp  \cong L^{p}\times L^{p} \times L^{p} \times \cdots
\end{equation*}
\end{conjecture}
\begin{conjecture}
Assume that $\product{M}=\infty$ and that the regularity assumptions of Theorem~\ref{prod-infty-main} are satisfied. Then
$$
\Scomp\cong L^{p}\times L^{p} \times L^{p} \times \cdots
$$
\end{conjecture}
It is entirely possible that one could weaken or drop some regularity assumptions, but on this point we feel that we should not say too much.

\bibliography{References}{}
\bibliographystyle{amsplain}

\end{document}